\documentclass[11pt]{amsart}
\usepackage{amsmath}
\usepackage[all]{xy}
\usepackage{amssymb}
\usepackage{mathrsfs}
\usepackage{cite}
\usepackage{hyperref}
\usepackage{amsfonts}
\usepackage{mathdots}
\theoremstyle{plain}
\newtheorem{thm}{Theorem}[section]
\newtheorem{prop}[thm]{Proposition}
\newtheorem{lem}[thm]{Lemma}
\newtheorem{cor}[thm]{Corollary}
\theoremstyle{remark}
\newtheorem{rem}[thm]{Remark}
\newtheorem*{acks}{Acknowledgements}
\theoremstyle{definition}
\newtheorem{defn}[thm]{Definition}

\newtheorem*{notn}{Notation}
\newtheorem{eg}[thm]{Example}
\theoremstyle{conjecture}
\newtheorem{conj}[thm]{Conjecture}
\newtheorem{ques}[thm]{Question}
\title{On nonexistence of semi-orthogonal decompositions in algebraic geometry}
\author{Xun Lin}
\email{lin-x18@mails.tsinghua.edu.cn}
\address{Yau Mathematical Sciences Center, Tsinghua University, Beijing, China.}
\begin{document}

\begin{abstract}
The nonexistence of semi-orthogonal decompositions in algebraic geometry is known to be governed by the base locus of the canonical bundle. We study another locus, namely the intersection of the base loci of line bundles that are isomorphic to the canonical bundle in the N\'{e}ron-Severi group, and show that it also governs the nonexistence of semi-orthogonal decompositions. As an application by using algebraically moving techniques, we prove that the bounded derived category of the $i$-th symmetric product of a smooth projective curve $C$ has no nontrivial semi-orthogonal decompositions when the genus $g(C)\geq 2$ and $i\leq g(C)-1$. We prove indecomposability of derived categories of some examples of elliptic surfaces with $P_{g}(X)=0$, and some natural examples of minimal surfaces of general type. Finally, an inequality involving phases of skyscraper sheaves for any Bridgeland stability condition is obtained.
\end{abstract}

\maketitle

\section{Introduction}
The bounded derived category of coherent sheaves of a variety encodes rich information of geometry of the variety. Some linear homological invariants can be directly read from the category, for example, the Hochschild homology and cohomology. One way to study derived categories is to decompose in sense of
semi-orthogonal decomposition. Naively, we would like to consider orthogonal decomposition, but for
the connected varieties, their bounded derived categories don't admit nontrivial orthogonal decompositions. However, there are many examples for the semi-orthogonal decompositions in algebraic geometry, for a survey of semi-orthogonal decompositions, see \cite{kuznetsov2015semiorthogonal}. Semi-orthogonal pieces of the derived categories of coherent sheaves are also main examples of noncommutative motives, and some classical
notions in algebraic geometry can be extended to this noncommutative setting. For example, many famous conjectures admit their noncommutative counterparts including Weil conjecture, Grothendieck conjecture of type $C$ and $D$ \cite{tabuada2019noncommutative}, and even Hodge conjecture, see A.\ Perry\cite{perry2020integral} firstly for the cases of the admissible subcategories of the bounded derived categories of coherent sheaves, and then \cite{lin2021noncommutative} for generalization to noncommutative motives.

\par

However, there are many examples of varieties whose derived categories are indecomposable. The indecomposability is closely related to birational geometry, and it is interesting to determine whether for a given variety, the bounded derived category admits no nontrivial semi-orthogonal decompositions. Given a smooth projective variety $X$, and blow up to $Y$, $D(Y)$ can be decomposed as $D(X)$ and several pieces from the centers\cite{Orlov}. It is natural to think that the minimal variety is indecomposable, but it is not true. For example, the Enriques surface has a nontrivial semi-orthogonal decomposition. However, if we assume that the canonical bundle is effective, there are many examples of minimal varieties
whose derived categories are indecomposable, for example, the minimal surface of Kodaira dimension $1$, and some examples of minimal surfaces of general type \cite[Section 4]{kawatani2018nonexistence}. We can make the following conjecture, which was stated for example in \cite[Conjecture 1.6]{biswas2020semiorthogonal} and \cite[Question E]{bastianelli2020indecomposability}.

\begin{conj}\label{minimalconj}
Let $X$ be a smooth projective variety. Assume the canonical bundle $K_{X}$ is nef and effective, then $D(X)$ has no non-trivial semi-orthogonal decompositions.
\end{conj}

A long-term project in the realm of derived categories in algebraic geometry is to make the semi-orthogonal decompositions of derived categories compatible with the minimal model program. For example, the $DK$ hypothesis. A good survey for this direction is \cite{kawamata2017birational}. Therefore, we can also ask the converse question.

\begin{conj}\cite[Conjecture 1.1]{kawatani2018nonexistence}
Let $X$ be a smooth projective variety. If $D(X)$ has no nontrivial semi-orthogonal decompositions, then $X$ is minimal.
\end{conj}

\begin{rem}
We should not expect to obtain $P_{g}(X)\neq 0$ since if $X$ is a bielliptic surface of Kodaira dimension $0$, then $P_{g}(X)=0$, but $D(X)$ admits no nontrivial semi-orthogonal decompositions \cite[Proposition 4]{kawatani2018nonexistence}.
\end{rem}

In this paper, we provide a general theory for the nonexistence of semi-orthogonal decompositions in Section $3$, which should be regarded as a mild generalization of the nonexistence theory in paper \cite{kawatani2018nonexistence}.

\begin{thm}\label{Picsod}(= Corollary \ref{refinesod})
Let $X$ be a smooth projective variety, define the new base points of the canonical bundle,
$$Z=\mathsf{PBs}\vert\omega_{X}\vert:=\cap_{L\in Pic^{0}(X)}\mathsf{Bs}\vert \omega_{X}\otimes L\vert.$$
If assume a semi-orthogonal decomposition, $D(X)=\langle \mathcal{A},\mathcal{B}\rangle$, then one of the following is true.

\begin{enumerate}
    \item For any $x\in X\setminus Z$, $k(x)\in \mathcal{A}$. In this case, support of all objects in $\mathcal{B}$ is contained in $Z$.
    \item For any $x\in X\setminus Z$, $k(x)\in \mathcal{B}$. In this case, support of all objects in $\mathcal{A}$ is contained in $Z$.
\end{enumerate}
\end{thm}

\begin{rem}\
\begin{enumerate}
   \item In the paper \cite{kawatani2018nonexistence}, the base locus of $\omega_{X}$ in $Pic(X)$ governs the non-existence of nontrivial semi-orthogonal decompositions. $Pic(X)$ is too large to consider. The theorem tells us that intersection of the base locus of class $\omega_{X}$ in $NS(X)=Pic(X)/Pic^{0}(X)=H^{1,1}(X,\mathbb{Z})$ is relevant to indecomposability of derived categories, though of course not easy to compute in general. It provides a bridge from the numerical geometry up to torsion to indecomposability of derived categories in algebraic geometry, namely the numerical equivalence classes of line bundles $Num(X)$ is the torsion-free part of $NS(X)$.
    \item We point out that it is impossible to generalize to $Num(X)$ in a naive way. For example, consider the Enriques surface $X$, then $\omega_{X}$ is numerically trivial, but $D(X)$ admits a nontrivial semi-orthogonal decomposition.
\end{enumerate}
\end{rem}

\begin{cor}(= Theorem \ref{emptynosod}) \label{emptynosodA}
Let $X$ be a smooth projective variety. Assume $Z=\mathsf{PBs}\vert\omega_{X}\vert$ is empty or $\dim Z=0$, then $D(X)$ has no nontrivial semi-orthogonal decompositioins.
\end{cor}

Using this theory, in Section $4$, we study elliptic surfaces of Kodaira dimension $1$ such that $P_{g}(X)=0$. When the canonical bundle has nontrivial sections, it is
proved in \cite[Section 4]{kawatani2018nonexistence} that $D(X)$ has no nontrivial semi-orthogonal decompositions. The base locus of the canonical bundle is the whole variety in our case. However, since we consider intersection base locus of line bundles that is the class of canonical bundle instead of base locus of the canonical bundle, we can deal with the case $P_{g}(X)=0$.

\begin{thm}\label{elliptic2A}(= Theorem \ref{elliptic2})
Let $\pi: X\rightarrow C$ be elliptic fibration, $X$ is minimal of Kodaira dimension $1$. Assume $P_{g}(X)=0$, $L:=R\pi^{1}\mathcal{O}(X)^{-1}\neq \mathcal{O}_{C}$, $g(C)=1$, $q(X)=1$, and $d=\mathsf{deg} L=0$, then $D(X)$ has no nontrivial
semi-orthogonal decompositions.
\end{thm}

We study some natural examples of minimal surfaces of general type.
\begin{thm}\label{surfacegeneraltypeA}(= Theorem \ref{surfacegeneraltype})
Let $X$ be a minimal surface of general type such that $P_{g}(X)=q(X)=1$, $K^{2}=2$ or $K^{2}=3$, then $D(X)$ has no nontrivial semi-orthogonal decompositions.
\end{thm}

Using the algebraically moving techniques, we also prove the following theorems.
\begin{thm}(= Theorem \ref{symmetriccurve})\label{thmA}
Let $X= S^{i}C$, the $i$-th symmetric product of a smooth projective curve $C$. Then $D(S^{i}C)$ has no nontrivial semi-orthogonal decompositions when $g(C)\geq 2$, and
$i\leq g(C)-1$.
\end{thm}

\begin{rem}\
\begin{enumerate}
    \item Some partial results of this theorem are obtained by several people, see \cite[Corollary D]{bastianelli2020indecomposability} and \cite{biswas2020semiorthogonal}, but no complete answers. Note that $S^{i}C$ in this theorem satisfies the assumption in Conjecture \ref{minimalconj}, see \cite{biswas2020semiorthogonal}.
    \item Note that $D(S^{i}C)$ considered in this theorem are conjecturally the building blocks of the derived category of moduli spaces of rank two bundles with fixed determinant on $C$ (this was conjectured by Narasimhan and independently in \cite{belmans2020graph}). The theorem tells us that these blocks are actually indecomposable.
    \item If $i\geq g$, the proof does not work. The essential reason is that $g-1-i<0$. We actually prove that $\mathsf{PBs}\vert \omega_{S^{i}C}\vert=\varnothing$, this would imply that $D(S^{i}C)$ has no nontrivial semi-orthogonal decompositions by general theory in the paper. For details, see the proof of Theorem \ref{symmetriccurve}.
    \item There is a semi-orthogonal decomposition for $S^{i}C$ if $i=g,g+1,...,2g-2$, see \cite[Corollary 5.12]{toda2019semiorthogonal}, or independently \cite[Theorem D]{belmans2019derived}, and \cite[corollary 3.8]{jiang2019derived}.
$$D(S^{i}C)=\langle J(C),\cdots, J(C),D(S^{2g-2-i}C)\rangle.$$
There are $i-g+1$ pieces of $J(C)$. As for $i\geq 2g-1$, $S^{i}C$ is a projective
bundle of $J(C)$. We complete the picture of indecomposability of bounded derived categories of symmetric product of curves.
\end{enumerate}
\end{rem}

\begin{thm}(=Theorem \ref{blockrank3bundle})
Let $C_{i}$ be curves of genus $g_{i}\geq 2$, and $A_{k}$ be an abelian variety. Let $X= \prod_{i}S^{j_{i}}C_{i}\times \prod_{k}A_{k}$, here $j_{i}\leq g_{i}-1$. Then, $D(X)$ has no nontrivial semi-orthogonal decompositions.
\end{thm}

\begin{rem}
The type of the varieties $X$ in the corollary are the conjecturally building blocks of motivic decomposition of moduli space of rank $r$ bundles with fixed determinant \cite[Conjecture 1.8]{gomez2020motivic}, and then the derived categories in the corollary are conjecturally building blocks of the derived category of rank $r$ vector bundles with fixed determinant.
\end{rem}

Recently D.\ Pirozhkov proposed a notion called stably semi-orthogonal indecomposable varieties, it is a more general notion than the indecomposibility, and stably semi-orthogonal indecomposability implies indecomposability\cite{pirozhkov2021stably}.
\begin{thm}(= Theorem \ref{stablyindecom})
Let $C$ be a smooth projective curve, $g(C)\geq 2$ and $i\leq g(C)-1$, then $X=S^{i}C$ is stably semi-orthogonal indecomposable if and only if $\mathsf{Bs}\vert \omega_{S^{i}C}\vert =\varnothing$.
\end{thm}

\begin{rem}
For $X=S^{i}C$, $g(C)\geq 2$, $i\leq g(C)-1$, two notions stably semi-orthogonal indecomposability and indecomposability can be distinguished by $\mathsf{PBs}\vert \omega_{X}\vert$ and $\mathsf{Bs}\vert \omega_{X} \vert$. It is interesting to find more  examples.
\end{rem}

In Section $5$, we give an interesting application to the Bridgeland stability conditions. It is closely related to a notion called geometric stability conditions, see for example\cite[Definition 2.8]{fu2021stability}.
\begin{thm}(= Theorem \ref{boundedphase})
Let $Z=\mathsf{PBs}\vert \omega_{X}\vert$.
Let $\sigma$ be a Bridgeland stability condition of $D(X)$. Take any $k(x)$, $x\in X\setminus Z$. Assume $k(x)$ is not semi-stable, then the phase number of $HN$ factors of $k(x): \phi_{1}>\phi_{2}>\cdots >\phi_{n}$ satisfies
$$\phi_{i}-\phi_{i+1}\leq \dim X-1.$$
\end{thm}

Lastly, observe that if the $Pic^{0}(X)$ is larger, it is more possible to get that $D(X)$
admits no nontrivial semi-orthogonal decompositions. We propose a refined version of Conjecture \ref{minimalconj} that will probably be easier to obtain.

\begin{conj}\label{newconj}
Let $X$ be a smooth projective variety. Assume $K_{X}$ is nef and effective, and $q(X)=h^{1}(\mathcal{O}_{X})\geq 1$ (this implies $\dim Pic^{0}(X)\geq 1$). Then $D(X)$ admits no nontrivial semi-orthogonal decompositions.
\end{conj}

\begin{rem}\
\begin{enumerate}
    \item $S^{i}C$ satisfies the assumptions in the conjecture when $i\leq g(C)-1$ and $g(C)\geq 2$.
    \item If $q(X)=0$, then
    $$\mathsf{Bs}\vert\omega_{X}\vert=\mathsf{PBs}\vert\omega_{X}\vert.$$
    The criteria of Corollary \ref{emptynosodA} is the same as the one of Corollary 1.5 in \cite{kawatani2018nonexistence}, but there are other techniques to deal with these cases. It is to consider the semi-orthogonal decompositions varied in a family\cite{bastianelli2020indecomposability}.
    For example, the Horikawa surfaces and the Ciliberto double planes are all of irregularity $0$, and in some cases, the base loci of the canonical bundles will be curves, but their bounded derived categories don't admit nontrivial semi-orthogonal decompositions by the well-behaviour of semi-orthogonal decompositions in families \cite[Section 4]{bastianelli2020indecomposability}.
\end{enumerate}
\end{rem}

\begin{notn}
$X$ is always assumed to be smooth projective variety over $\mathbb{C}$. We write $D(X)$ as the bounded derived category of coherent sheaves of $X$, and $D_{qch}(X)$ as the derived category of $\mathcal{O}_{X}$ modules with quasi-coherent cohomology. We write $P_{g}(X)=h^{0}(X,\omega_{X})$, $q(X)=h^{1}(X,\mathcal{O}_{X})$, and $K^{2}=c_{1}(X)^{2}$.
\end{notn}

\begin{acks}
The author would like to thank Arend Bayer, Pieter Belmans, Will Donovan, Qingyuan Jiang, Shinnosuke Okawa, XueQing Wen for helpful conversations of the related topics. The author thanks ShiZhuo Zhang for pointing out a simpler proof of Theorem \ref{pic0sod}, see Remark~\ref{trickz}.
\end{acks}

\section{Semi-orthogonal decomposition}
\subsection{Semi-orthogonal Decomposition}
 \begin{defn}\label{defnsod}
  Let $\mathcal{T}$ be a $k$-linear triangulated category. We say a collection of sub-triangulated categories $\{\mathcal{A}_{1},\mathcal{A}_{2},\cdots,\mathcal{A}_{n}\}$ of $\mathcal{T}$ is a semi-orthogonal decomposition of $\mathcal{T}$ if the following holds.

  \begin{enumerate}
    \item For $i>j$, we have $Hom(\mathcal{A}_{i},\mathcal{A}_{j})=0$.
    \item Given any object $E\in\mathcal{T}$, there is a sequence of morphisms
        $$\xymatrix{E_{n+1}=0\ar[r]&E_{n}\ar[r]&E_{n-1}\ar[r]&E_{n-2}\ar[r]&\cdots\ar[r]&E_{2}\ar[r]&E_{1}=E}$$
    such that the cone of $E_{i}\rightarrow E_{i-1}$ belongs to $\mathcal{A}_{i-1}$.
  \end{enumerate}
   We write the semi-orthogonal decomposition as $\mathcal{T}=\langle \mathcal{A}_{1},\mathcal{A}_{2},\cdots,\mathcal{A}_{n}\rangle$.
 \end{defn}

  \begin{rem}
   We often call this version of semi-orthogonal decomposition the (weak) semi-orthogonal decomposition. We say strong semi-orthogonal decomposition if $\mathcal{A}_{i}$ is an admissible subcategory, that is, the embedding into $\mathcal{T}$ has both left and right adjoint.
   \end{rem}

 \begin{lem}\label{unique}
 Let $\mathcal{T}=\langle \mathcal{A},\mathcal{B}\rangle$ be a semi-orthogonal decomposition.
 Let $E\in \mathcal{T}$, consider a triangle from the semi-orthogonal decomposition $$b\rightarrow E\rightarrow a\rightarrow b[1].$$
 Suppose there is an another triangle
 $$b'\rightarrow E \rightarrow a' \rightarrow b'[1]$$
 such that $b'\in \mathcal{B}$, $a'\in \mathcal{A}$, then there exist a unique
 isomorphism $g$ from $b$ to $b'$, and also a unique isomorphism $f$ from $a$ to $a'$ such that the following diagram is commutative.
 $$\xymatrix{b\ar[r]^{i}\ar[d]_{g}&E\ar[r]^{L}\ar[d]^{id}&a\ar[r]\ar[d]_{f}&b[1]\ar[d]_{g[1]}\\
 b'\ar[r]^{i'}&E\ar[r]^{L'}&a'\ar[r]&b'[1]}$$
 \end{lem}

 \begin{proof}
 Apply $Hom(b,\bullet)$ to the second triangle, we get isomorphism
 $Hom(b,b')\cong Hom(b,E)$. Take element $id\circ i\in Hom(b,E)$. So we get a unique element $g$ with diagram commutative. Similar to the morphism $f$. To show that $g$ and $f$ are isomorphism, we apply the inverse direction to obtain morphism $g'$ and $f'$. Since eventually, the identity morphism is the unique element which transforms the triangle to the same triangle, $g'$ and $f'$ are the inverses of
 $g$ and $f$ respectively.
 \end{proof}

\begin{lem}\label{pic0invariant1} \cite[Theorem 3.9]{kawatani2018nonexistence}
Let $X$ be a projective scheme. Assume $Perf(X)=\langle \mathcal{A},\mathcal{B}\rangle$, then for any $L\in Pic^{0}(X)$, we have
$\mathcal{A}\otimes L=\mathcal{A}$.
\end{lem}

\begin{rem}\label{tensorinvariant}
This would also imply $\mathcal{B}\otimes L=\mathcal{B}$ for any $L\in Pic^{0}(X)$. The proof is easy:
consider any object $E\otimes L\in \mathcal{B}\otimes L$, where $E\in \mathcal{B}$. Consider the triangle with respect to the semi-orthogonal decomposition
$$b\rightarrow E\otimes L\rightarrow a\rightarrow b[1].$$
Tensor by line bundle $L^{-1}$, clearly $a\otimes L^{-1}\in \mathcal{A}$. We get a new triangle
$$b\otimes L^{-1}\rightarrow E\rightarrow a\otimes L^{-1}\rightarrow b\otimes L^{-1}[1]$$
Since $Hom^{\ast}(E,a\otimes L^{-1})=0$, therefore $b\otimes L^{-1}\cong E\oplus a\otimes L^{-1}[-1]$. Since $Hom^{\ast}(b\otimes L^{-1},a\otimes L^{-1})=0$, therefore we get $a\cong 0$. Thus, $E\otimes L\cong b\in \mathcal{B}$.
\end{rem}

\begin{lem}\label{pic0invariant2}
Let $X$ be a smooth projective variety. Assume $D(X)=\langle \mathcal{A},\mathcal{B}\rangle$. Take any skyscraper sheaf
$k(x)$, consider the triangle $b\rightarrow k(x)\rightarrow a\rightarrow b[1]$ from
the semi-orthogonal decomposition. Then for any line bundle $L\in Pic^{0}(X)$, we have $b\otimes L\cong b$, $a\otimes L\cong a$.
\end{lem}

\begin{proof}
Tensor with $L$, we get another triangle, $b\otimes L\rightarrow k(x)\rightarrow a\otimes L\rightarrow b\otimes L[1]$. By Lemma \ref{pic0invariant1}, $b\otimes L \in \mathcal{B}$, and $a\otimes L\in \mathcal{A}$. Hence, by uniqueness of the triangle from the semi-orthogonal decomposition, see Lemma \ref{unique}, $b\otimes L\cong b$, $a\otimes L\cong a$
\end{proof}

\section{General theory for the nonexistence of semi-orthogonal decompositions}
\subsection{First generalization}
\begin{prop}\label{sodsupport}(\cite[Thm 3.1]{kawatani2018nonexistence})
  Let $X$ be a smooth projective variety, $Z=\mathsf{Bs}\vert \omega_{X}\vert$. If assume a semi-orthogonal decomposition, $D(X)=\langle \mathcal{A},\mathcal{B}\rangle$, then one of the following is true.
\begin{enumerate}
    \item For any $x\in X\setminus Z$, $k(x)\in \mathcal{A}$. In this case, support of all objects in $\mathcal{B}$ is contained in $Z$.
    \item For any $x\in X\setminus Z$, $k(x)\in \mathcal{B}$. In this case, support of all objects in $\mathcal{A}$ is contained in $Z$.
\end{enumerate}
\end{prop}

We give a generalization of the Proposition \ref{sodsupport} for objects which are not skyscraper sheaves.
\begin{thm}\label{generalsodsupport}
  Let $X$ be a smooth projective variety, and a semi-orthogonal decomposition, $D(X)=\langle\mathcal{A},\mathcal{B}\rangle$. Let $E$ be an object supported in a closed subset $Z$ such that $Hom(E,E)=k$. Suppose $\exists$ a section $s$ $\in H^{0}(X,\omega_{X})$ such that $Z\cap Z(s)= \varnothing$. Then $E\in \mathcal{A}$ or $\mathcal{B}$. Note that $k(x)$ with $x\in X\setminus \mathsf{Bs}\vert\omega_{X}\vert$ satisfies the assumption.
\end{thm}

\begin{proof}
 Consider a triangle $b\rightarrow E\rightarrow a\rightarrow b[1]$ corresponding to the semi-orthogonal decomposition. By assumption, we can choose a section $s$ such that $Z(s)\cap Z=\varnothing$. Write $U=X\setminus Z(s)$, $j:U \hookrightarrow X$. Then compose the section $s$ with morphism $a\rightarrow b[1]$, we get a morphism $a\rightarrow b\otimes \omega_{X}[1]$. But according to the semi-orthogonality, $Hom_{X}(a,b\otimes \omega_{X}[1])\cong Hom_{X}(b,a[\dim X-1])^{\vee}= 0$.
 Hence the morphism $a\rightarrow b[1]$ pulling back to $U$ becomes zero because $s|_{U}$ is invertible. Now pulling back the triangle to $\mathsf{D}^{\mathsf{b}}(U)$, we have $j^{\ast}b\rightarrow j^{\ast}E\rightarrow j^{\ast}a\rightarrow j^{\ast}b[1]$, hence $j^{\ast}E\cong j^{\ast}a\oplus j^{\ast}b$. By adjunctions, we have $Hom_{U}(j^{\ast}E,j^{\ast}E)\cong Hom_{X}(E,Rj_{\ast}j^{\ast}E)$.
 Since $Z\cap Z(s)=\varnothing$, the support of $E$ is contained in $U$. There is a triangle in $D_{qch}(X)$,
 $$R\underline{\Gamma_{Z(s)}}E\rightarrow E\rightarrow Rj_{\ast}j^{\ast}E \rightarrow R\underline{\Gamma_{Z(s)}}[1].$$
 Since $R\underline{\Gamma_{Z(s)}}E= 0$, hence $E\cong Rj_{\ast}j^{\ast}E$. Therefore,
 $$Hom_{U}(j^{\ast}E,j^{\ast}E)\cong Hom_{X}(E,E)\cong k.$$
 Back to equality $j^{\ast}E\cong j^{\ast}a\oplus j^{\ast}b$, the equation above implies that $j^{\ast}a\cong 0$ or $j^{\ast}b\cong 0$. Suppose $j^{\ast}a\cong 0$, it implies that the support of $a$ is in $Z(s)$. Clearly $R\mathcal{H}om(E,a)\cong 0$. Applying $R\Gamma$ and taking homology, we have $Hom^{\ast}_{X}(E,a)\cong 0$, which implies $b\cong a[-1]\oplus E$. Therefore $a\cong 0$, and $b\cong E$. Suppose $j^{\ast}b\cong 0$, then $E\cong a$ follows from the same reason.
\end{proof}

 \subsection{Generalization to the intersection of base loci of line bundles that are isomorphic to the canonical bundle in N\'{e}ron-severi group}

 \begin{thm}\label{pic0sod}
Let $X$ be a smooth projective variety, $Z=\mathsf{Bs}\vert \omega_{X}\otimes L\vert$, $L\in Pic^{0}(X)$. If assume a semi-orthogonal decomposition, $D(X)=\langle \mathcal{A},\mathcal{B}\rangle$, then one of the following is true.

\begin{enumerate}
    \item For any $x\in X\setminus Z$, $k(x)\in \mathcal{A}$. In this case, support of all objects in $\mathcal{B}$ is contained in $Z$.
    \item For any $x\in X\setminus Z$, $k(x)\in \mathcal{B}$. In this case, support of all objects in $\mathcal{A}$ is contained in $Z$.
\end{enumerate}
 \end{thm}

 \begin{proof}
 Take $x\in X\setminus Z$, we prove that $k(x)\in \mathcal{A}\ or\ \mathcal{B}$.
 Take a section $s\in \Gamma(X, \omega_{X}\otimes L)$ such that $s(x)\neq 0$. Consider the triangle $b\rightarrow k(x)\rightarrow a\rightarrow b[1]$ from the semi-orthogonal decomposition. Consider the composition of $a \rightarrow b[1]$ and  $b[1]\rightarrow b[1]\otimes\omega_{X}\otimes L$ via section $s$. According to Lemma \ref{pic0invariant2}, $b\otimes L\cong b$, hence the composition morphism $a\rightarrow b[1]\otimes \omega_{X}\otimes L$ is zero by Serre duality and semi-orthogonality. Define $U=X\setminus Z(s)$, write $j:U\hookrightarrow X$.
 Then pulling pack to $j$, there is a new triangle
 $$j^{\ast}b\rightarrow k(x)\rightarrow j^{\ast}a\rightarrow j^{\ast}b[1].$$
 But since section $s$ is invertible in $U$, the morphism $j^{\ast}a\rightarrow j^{\ast}b[1]$ is trivial. Therefore, $k(x)\cong j^{\ast}a\oplus j^{\ast}b$. Since $k(x)$ is simple, $j^{\ast}a=0$, or $j^{\ast}b=0$. Without loss of generality, we assume $j^{\ast}a\cong 0$, which implies that support of $a$ is in $Z(s)$. However, since $x\in U$, the morphism $k(x)[m]\rightarrow a[m]$ must be zero. Then, $b\cong a[-1]\oplus k(x)$. Thus $a=0$.
 \par
 The support statement follows easily from the statement of the skyscraper sheaves since the skyscraper sheaves can only belong to exactly one component, see proof of Theorem 1.2\cite[Theorem 1.2]{kawatani2018nonexistence}.
 \end{proof}

 \begin{rem}\ \label{trickz}
 \begin{enumerate}
     \item In fact, we don't need the fact $\mathcal{B}\otimes L\subset \mathcal{B}$ to prove our theorem. Since $a\otimes L^{-1}\in \mathcal{A}$, we have
     $$Hom^{\ast}(a, b\otimes \omega_{X}\otimes L)\cong Hom^{\ast}(a\otimes L^{-1}, b\otimes \omega_{X})=0.$$
     We thank Shizhuo Zhang for pointing out this simpler proof to the author.
     \item One may wonder to generalize Theorem \ref{generalsodsupport}. We need to assume that $E$ is  preserved by the $Pic^{0}(X)$.
 \end{enumerate}

 \end{rem}

 \begin{cor}\label{refinesod}
Let $X$ be a smooth projective variety, define the base points of the paracanonical systems, $$Z=\mathsf{PBs}\vert\omega_{X}\vert:=\cap_{L\in Pic^{0}(X)}\mathsf{Bs}\vert \omega_{X}\otimes L\vert.$$
If assume a semi-orthogonal decomposition, $D(X)=\langle \mathcal{A},\mathcal{B}\rangle$, then one of the following is true.

\begin{enumerate}
    \item For any $x\in X\setminus Z$, $k(x)\in \mathcal{A}$. In this case, support of all objects in $\mathcal{B}$ is contained in $Z$.
    \item For any $x\in X\setminus Z$, $k(x)\in \mathcal{B}$. In this case, support of all objects in $\mathcal{A}$ is contained in $Z$.
\end{enumerate}

 \end{cor}
\begin{proof}
Let $x\notin Z$, then there is a $L\in Pic^{0}(X)$ such that $x\notin \mathsf{Bs}\vert \omega_{X}\otimes L\vert$, according to Theorem \ref{pic0sod}, $k(x)$ belongs to $\mathcal{A}$ or $\mathcal{B}$, the statements of the Corollary follows.
\end{proof}
 \begin{rem}\
 \begin{enumerate}
     \item The base locus of algebraic class [$\omega_{X}$] in N\'{e}ron-Severi group  $Pic(X)/Pic^{0}(X)$ would be counted to the problem of semi-orthogonal decomposition. Thus, if we have good knowledge of N\'{e}ron-Severi group and the class [$\omega_{X}$], it is more possible to determine whether the derived categories have nontrivial semi-orthogonal decomposition.
     \item We point out that it is impossible to generalize to $Num(X)$ in a naive way. For example, consider the Enriques surface $X$. $\omega_{X}$ is numerically trivial, but $D(X)$ admits a nontrivial semi-orthogonal decomposition. However, it is possible if we first have the intersection base locus of line bundles of canonical class in $NS(X)$ is a proper closed subset.
 \end{enumerate}
 \end{rem}

 \begin{eg}
 Let $A$ be an abelian vareity. Then $Pic^{0}(A)$ is exactly the translation invariant line bundle in $A$. Let $L\in Pic^{0}(A)$ which is not trivial, we have
 $H^{0}(A,L)=0$. Hence $\mathsf{Bs}\vert\omega_{A}\vert=\varnothing$, but $\mathsf{Bs}\vert L\vert=A$. The base locus in an algebraic class can change.
 \end{eg}

 \begin{thm}\label{emptynosod}
Let $X$ be a smooth projective variety, $Z=\mathsf{PBs}\vert\omega_{X}\vert$. Suppose $Z=\varnothing$ or $\dim Z=0$, then $D(X)$ have no nontrivial semi-orthogonal decompositions.
 \end{thm}

 \begin{proof}
 Let $D(X)=\langle \mathcal{A},\mathcal{B}\rangle$, then according to Corollary \ref{refinesod}, we can assume that all objects of $\mathcal{A}$ support in $Z$.
 If $Z=\varnothing$, then $\mathcal{A}=0$.
 \par
 If $\dim Z=0$, then, $\mathcal{A}\otimes \omega_{X}=\mathcal{A}$. By semi-orthogonality and Serre duality, the original semi-orthogonal decomposition is an orthogonal decomposition, hence must be trivial since $X$ is connected.
 \end{proof}

 \begin{thm}\label{product1}
 Let $X_{i}$ be a finite collection of the varieties with the property that
 $\mathsf{PBs}\vert \omega_{X_{i}}\vert=\varnothing$. Let $Y=X_{1}\times X_{2}\times \cdots \times X_{n}$. Then $D(Y)$ have no nontrivial semi-orthogonal decompositions.
 \end{thm}

 \begin{proof}
 We prove that $\mathsf{PBs}\vert \omega_{Y}\vert=\varnothing$. Take any closed point $z=(z_{1}, \cdots, z_{n})\in Y$.
 Since $\mathsf{PBs}\vert \omega_{X_{i}}\vert=\varnothing$, there exist effective divisors $D_{i}$ such that $z_{i}\notin D_{i}$, and $D_{i}$ is algebraic equivalent to $K_{X_{i}}$. As an algebraic class,
 $K_{Y}=\sum X_{1}\times X_{2}\times \cdots \times D_{i}\times\cdots\times X_{n}$, we are done.
 \end{proof}

 \subsection{Induction process}
 Inspired by proof of Theorem \ref{pic0sod}, we provide an induction process to
 reduce the base points $\mathsf{PBs}\vert \omega_{X}\vert$.

 \begin{thm}\label{inductionprocess}
Let $D(X)=\langle \mathcal{A},\mathcal{B}\rangle$ be a semi-orthogonal decomposition, $Z=\mathsf{PBs}\vert \omega_{X}\vert$.
Let $M$ be a line bundle locally trivial around $Z$, then
 any skyscraper sheaf $k(x)$ with $x\notin Z'=\mathsf{Bs}\vert \omega_{X}\otimes M\vert$ must belong to $\mathcal{A}$ or $\mathcal{B}$.
 \end{thm}

 \begin{proof}
 Observe that objects of $\mathcal{A}$ or $\mathcal{B}$ are supported in $Z$. Without loss of generality, assume objects of $\mathcal{A}$ are supported in $Z$. Since $M$ is locally trivial around $Z$, take any object $E\in \mathcal{A}$, then $E\otimes M\cong E$. The trick is as follows: we only need to consider the point $x\in Z\setminus Z'$.
 \par
 Let $b\rightarrow k(x)\rightarrow a\rightarrow b[1]$ be the triangle from the semi-orthogonal decomposition. Tensor by line bundle $M$, there is a new triangle
 $$b\otimes M\rightarrow k(x)\otimes M\rightarrow a\otimes M\rightarrow b[1].$$
 By assumption, there is an open neighborhood $Z\subset U$ such that $M$ is trivial on $U$. Hence $a\otimes M\cong a\otimes M\vert_{U}\cong a$, as support of $a$ is in $Z$.

 Since $\mathcal{A}\otimes M=\mathcal{A}$, using the similar argument in proof of Remark \ref{tensorinvariant}, we get $\mathcal{B}\otimes M=\mathcal{B}$. Using the same techniques as in proof of Theorem \ref{pic0sod}, we prove that $k(x)$ belong to $\mathcal{A}$ or $\mathcal{B}$.
 \par
 In fact, the trick in Remark \ref{trickz} (1) can also apply directly.
 \end{proof}

 \begin{rem}
      Actually, we just need to assume that $M$ restricts to $nZ_{red}$ is trivial for any integer~$n$. The proof is the same, we can use that if $E$ is supported in $Z$, then there is an $E_{n}\in D(nZ_{red})$ and $i_{n}: nZ_{red}\hookrightarrow X$ such that $i_{n,\ast}E_{n}=E$. Thus, $E\otimes M\cong E$ by projection formula.
 \end{rem}

  \begin{eg}
 Let $f:X\longrightarrow Y$ be a morphism that contracts $Z$ to a point $y$, where
 $Z=\mathsf{PBs}\vert \omega_{X}\vert$. Since any line bundle $N$ is trivial around the point $y$, hence $M:=f^{\ast}N$ is a line bundle that is trivial around $Z$.
 \end{eg}

  If we can choose a line bundle $M$ locally trivial around $Z$ such that $Z_{1}:=Z\setminus \mathsf{Bs}\vert\omega_{X}\otimes M\vert$ becomes smaller, then apply the same process for $Z_{1}$. We get an induction process to prove that
 $D(X)$ have no nontrivial semi-orthogonal decompositions. It is interesting to find such an example to work.

\subsection{The singular varieties}\
 Let $X$ to be a variety with Cohen-Macaulay singularities. The Proposition \ref{sodsupport} was generalized to such $X$ by Dylan Spence \cite{spence2021note}. The theorems in this section can be  generalized to the cases of varieties with Cohen-Macaulay singularities in paper \cite{spence2021note}. For completeness, we state the theorem below.

\begin{thm}\label{singularsod}
Let $X$ be a Cohen-Macaulay projective variety with dualizing sheaf $\omega_{X}$. Assume a semi-orthogonal decomposition $Perf(X)=\langle \mathcal{A},\mathcal{B}\rangle$.
$Z=\mathsf{PBs}\vert \omega_{X}\vert$. Then one of the following is true.

 \begin{enumerate}
    \item  For any $x\in X\setminus Z$, $\mathcal{O}_{Z_{x}}\in \mathcal{A}$. In this case, support of all objects in $\mathcal{B}$ is contained in $Z$.
    \item For any $x\in X\setminus Z$, $\mathcal{O}_{Z_{x}}\in \mathcal{B}$. In this case, support of all objects in $\mathcal{A}$ is contained in $Z$.
 \end{enumerate}
$\mathcal{O}_{Z_{x}}$ is the sheaf of Koszul zero cycle, and the base locus of sheaves is also defined in \cite{spence2021note}.
 \end{thm}

\begin{proof}
Since the Koszul zero cycle $\mathcal{O}_{Z_{x}}$ is a perfect complex that is supported in closed point $x$, therefore $\mathcal{O}_{Z_{x}}\otimes L\cong \mathcal{O}_{Z_{x}}$ for any $L\in Pic^{0}(X)$. The remaining proof is the same as the proof of \cite[Theorem 3.1]{spence2021note}.
\end{proof}

 \section{Examples}

 \subsection{Derived category of symmetric product of curves}\quad \par
 In this subsection, we prove that the derived category of $i$-th symmetric product of a curve has no nontrivial semi-orthogonal decompositions when genus $g$ of the curve is greater or equal to $2$, and $i\leq g-1$. Furthermore, as a by product, we prove that $D(\prod_{i}S^{j_{i}}C_{i})$ has no non-trivial semi-orthogonal decompositions, where $j_{i}\leq g(C_{i})-1$.
 \par

 To warm-up, we deal with a well-known example, to see how generalization in Section $3$ works.

 \begin{eg}
 Let $X=C_{1}\times \cdots \times C_{n}$, suppose the genus of $C_{i}$ is positive, then there are no nontrivial semi-orthogonal decompositions for $D(X)$.

 \par

  Clearly $K_{X}\cong K_{C_{1}}\boxtimes\cdots\boxtimes K_{C_{n}}$. As a class in N\'{e}ron-Severi group,
  $$K_{X}=(2g_{1}-2)p_{1}\times C_{2}\times \cdots\times C_{n} + \cdots + (2g_{n}-2)C_{1}\times \cdots\times C_{n-1}\times p_{n},$$
  where $p_{i}$ are choosing fixed points. Thus, consider any closed points $x=(x_{1},\cdots,x_{n})\in C_{1}\times\cdots \times C_{n}$, we can choose generic $p_{i}$ to make sure that $x\notin \cup_{i} C_{1}\times \cdots\times p_{i}\times\cdots \times C_{n}$. Hence, $\mathsf{PBs}\vert \omega_{X}\vert=\varnothing$. By Corollary \ref{refinesod}, $D(X)$ have no nontrivial semi-orthogonal decomposition.
 \end{eg}

 \begin{rem}\
 \begin{enumerate}
    \item Clearly $\mathsf{Bs}\vert \omega_{X}\vert=\varnothing$. So the result that $D(X)$ admits no non-trivial semi-orthogonal decompositions is already known to the experts. The new idea here is that every points in a curve are algebraically equivalent, so we can move the points.
    \item If one of the curves is $\mathbb{P}^{1}$, it is well known that we can
    decompose $D(X)$ via decomposing $D(\mathbb{P}^{1})$.
 \end{enumerate}
 \end{rem}

  \begin{thm}
  Let $G$ be an algebraic group, and $H$ be a hypersurface of $G$. Take any closed point $p\in H$, there always exists a generic $g\in G$ such that $p\notin gH$.
 \end{thm}

 \begin{proof}
 Without loss of generality, we assume $p=e$, the identity element of $G$. Define $Z=H^{-1}$, it is a proper closed subset of $G$. Take $g\notin Z$, then $e\notin gH$.
 \end{proof}

 \begin{cor}\label{translation}
 Let $D$ be a hypersurface of abelian variety $A$. Fix any point $p\in D$, there exists generic $a\in A$ such that $p\notin a+D$.
 \end{cor}

\begin{thm}\label{symmetriccurve}
Let $C$ be a smooth projective curve of genus $g(C)\geq 2$. $S^{i}C$ is the $i$ symmetric product of $C$.
Then $\mathsf{PBs}\vert \omega_{S^{i}C}\vert=\varnothing$ when $i\leq g(C)-1$. In particular, $D(S^{i}C)$ has no nontrivial semi-orthogonal decompositions in these cases.
\end{thm}

\begin{proof}
We mainly use the concrete description of class $\omega_{S^{i}C}$ in N\'{e}ron-Severi group of $S^{i}C$, namely $\omega_{S^{i}C}\cong^{alg} (g-i-1)x_{p_{0}}+ \theta$, see \cite[Lemma 2.1]{biswas2020semiorthogonal} or \cite[(14.5), (14.9)]{Mac}.
Here, $x_{p_{0}}$ is the image of
$$j: S^{i-1}C\hookrightarrow S^{i}C\quad\quad D\mapsto D+p_{0},$$
$p_{0}$ is a fixed point in $C$. $\theta$ is the pulled pack line bundle $u^{\ast}\mathcal{O}(\Theta)$.
$$u:S^{i}C\longrightarrow J(C)\quad\quad D\mapsto \mathcal{O}(D-dp_{0}).$$
$\Theta$ is the theta divisor of $J(C)$. Though we don't have a canonical choice of $\Theta$, we will fix a choice, for example $\Theta=u(S^{g-1}C)+b$ with a constant $b\in J(C)$ to avoid the confusion of pulling back divisors, but the pulling back of line bundle by $u$ always make sense.
\par
Clearly, when $i\leq g-1$, we have $g-i-1\geq 0$. Now take any closed point $z=z_{1}+z_{2}+\cdots+z_{i}\in S^{i}C$. We can choose a point
$p\in C$ such that $p\neq z_{i}$. Again we have isomorphism $\omega_{S^{i}C}\cong^{alg}(g-i-1)x_{p}+\theta$ since $x_{p}\cong^{alg}x_{p_{0}}$. Clearly, $z\notin x_{p}$ by our choice of point $p$. This means that we can move $x_{p_{0}}$ algebraically to $x_{p}$ which avoids $z$. The next step is to move $\Theta$ algebraically.
\par
If $u(z)\notin \Theta$, the canonical section $S_{(g-i-1)x_{p}+\theta}\in \Gamma(S^{i}C,\mathcal{O}((g-i-1)x_{p})\otimes u^{\ast}\mathcal{O}(\Theta))$ does not vanish at $z$. Hence, $z\notin \mathsf{PBs}\vert \omega_{S^{i}C}\vert$.
\par
If $u(z)\in\Theta$, we can move $\Theta$ algebraically to some $t_{a}^{\ast}\Theta$ by translation $t_{a}$ of $J(C)$. In fact, if $A$ is an abelian variety, then $Pic^{0}(A)$ is exactly the translation invariant line bundles, and $t^{\ast}_{a}\Theta-\Theta$ is translation invariant, see for example \cite[Section 9]{Hbook}. In particular, $u(z)\notin t_{a}^{\ast}\Theta$ via generic choice of $a\in J(C)$ by Corollary \ref{translation}. Therefore,
$$\omega_{S^{i}C}=^{alg}(g-i-1)x_{p}+u^{\ast}t_{a}^{\ast}\Theta,$$
and the canonical section of $(g-i-1)x_{p}+u^{\ast}t_{a}^{\ast}\Theta$ does not vanish at $z$. That is, $z\notin \mathsf{PBs}\vert \omega_{S^{i}C}\vert$.
\par
To sum up, $\mathsf{PBs}\vert \omega_{S^{i}C}\vert=\varnothing$. According to Theorem \ref{emptynosod}, $D(S^{i}C)$ has no nontrivial semi-orthogonal decompositions.
\end{proof}

\begin{rem}
When $i\geq g$, the proof does not work. Note that there is a semi-orthogonal decomposition for $S^{i}C$ if $i=g,g+1,...,2g-2$, see \cite[Corollary 5.12]{toda2019semiorthogonal} or independently \cite[Theorem D]{belmans2019derived}, and \cite[corollary 3.8]{jiang2019derived}.
$$D(S^{i}C)=\langle J(C),\cdots, J(C),D(S^{2g-2-i}C) \rangle.$$
There are $i-g+1$ pieces of $J(C)$. As for $i\geq 2g-1$, $S^{i}C$ is a projective
bundle of $J(C)$.
\end{rem}

\begin{cor}\label{product2}
Let $C_{i}$ be curves of genus $g_{i}\geq 2$. Let $X= \prod_{i}S^{j_{i}}C_{i}$, here $j_{i}\leq g_{i}-1$. Then $D(X)$ has no nontrivial semi-orthogonal decompositions.
\end{cor}

\begin{proof}
Combining Theorem \ref{symmetriccurve} and Theorem \ref{product1}.
\end{proof}

\begin{cor}\label{blockrank3bundle}
Let $C_{i}$ be curves of genus $g_{i}\geq 2$. Let $X= \prod_{i}S^{j_{i}}C_{i}$, here $j_{i}\leq g_{i}-1$. $Y=X\times \prod A_{j}$, where $A_{j}$ is an abelian variety. Then $D(Y)$ has no nontrivial semi-orthogonal decompositions.
\end{cor}

\begin{proof}
This is clear since $\mathsf{Bs}\vert \omega_{A_{j}}\vert=\varnothing$.
\end{proof}

\begin{rem}
The reason we state the obvious corollary here is that the varieties in the corollary are the conjecturally building blocks of motivic decomposition of  moduli space of rank $r$ bundles with fixed determinant \cite[Conjecture 1.8]{gomez2020motivic}, and then the derived categories in the corollary are conjecturally the building blocks of derived category of rank $r$ vector bundles with fixed determinant.
\end{rem}

There is a notion called stably semi-orthogonal indecomposable, which is more general than the notion of indecomposability, and stably semi-orthogonal indecomposability implies indecomposability \cite{pirozhkov2021stably}. It is natural to ask whether $S^{i}C$ is stably semi-orthogonal indecomposable when $i\leq g(C)-1$ and $g(C)\geq 2$.

\begin{thm}\label{stablyindecom}
Let $C$ be a smooth projective curve, $i\leq g(C)-1$ and $g(C)\geq 2$, then $S^{i}C$ is stably semi-orthogonal indecomposable if and only if $\mathsf{Bs}\vert K_{S^{i}C}\vert =\varnothing$.
\end{thm}

\begin{proof}
According to \cite[Proposition 3.4]{biswas2020semiorthogonal}, the Albanese map $u: S^{i}C\rightarrow J(C)$ is finite if and only if $\mathsf{Bs}\vert K_{S^{i}C}\vert=\varnothing$. If $u$ is finite, then according to \cite[Theorem 1.4]{pirozhkov2021stably}, $S^{i}C$ is stably semi-orthgonal indecomposable. If $u$ is not finite, then $S^{i}C$ will contain
a fiber which is a projective space $\mathbb{P}^{m}$, $m\geq 1$. Since $D(\mathbb{P}^{m})$ is decomposable, hence according to \cite[Corollary 2.5]{pirozhkov2021stably}, $S^{i}C$ will not be stably semi-orthogonal indecomposable.
\end{proof}

\subsection{Derived categories of elliptic surfaces of Kodaira dimension $1$}\
Let $X$ be a minimal projective smooth surface with Kodaira dimension $1$. It was proved that if $P_{g}(X)\geq1$, $D(X)$ have no nontrivial semi-orthogonal decompositions \cite[Theorem 4.2]{kawatani2018nonexistence}.  We give a further description of $D(X)$ for cases $P_{g}(X)=0$ in the sense whether it has nontrivial semi-orthogonal decompositions, which can not be achieved by just considering the base locus of the canonical bundle. Note that $X$ admits a fibration $\pi: X\rightarrow C$ with general fiber elliptic curve. $C$ is a smooth projective curve.
\begin{prop}\label{basicelliptic} \cite[Chap 7, Lemma 13, Lemma 14]{Fred}
Let $\pi: X\rightarrow C$ be an elliptic fibration, then $L:=R\pi^{1}\mathcal{O}(X)^{-1}$ is a line bundle on $C$. In particular, $d=\mathsf{deg} L\geq 0$, and $\mathcal{X}(\mathcal{O}_{X})=\mathsf{deg}(L)$.
\begin{enumerate}
   \item If $L\neq \mathcal{O}_{C}$, then $q(X):=h^{1}(X,\mathcal{O}_{X})=g(C)$ and
    $P_{g}(X)=d+g(C)-1$.
   \item If $L=\mathcal{O}_{C}$, then $q(X)=g(C)+1$ and $p_{g}(X)=g(C)$.
\end{enumerate}
\end{prop}

\begin{rem}\label{class0}
Assume $P_{g}(X)=0$. There are three possibilities.
\begin{enumerate}
    \item $L:=R\pi^{1}\mathcal{O}(X)^{\vee}=\mathcal{O}_{C}$, $g(C)=0$, and $q(X)=1$.
    \item $L\neq \mathcal{O}_{C}$,
$g(C)=1$, $q(X)=1$, and $d=0$.
    \item $L\neq \mathcal{O}_{C}$, $g(C)=0$, $q(X)=0$, and $d=1$. In this case, $\mathcal{O}_{X}$ is an exceptional object.
\end{enumerate}
\end{rem}

\begin{thm}\label{canonicalbundle}\cite[Chap 7, Theorem 15]{Fred}
Let $X$ be minimal, and $\pi: X\rightarrow C$ is an elliptic fibration. Let $F_{i}$ be the multiple fibers of $\pi$ with multiplicity $m_{i}\geq 2$. We have a formula for the canonical bundle of $X$,
$$\omega_{X}=\pi^{\ast}(K_{C}\otimes L)\otimes \mathcal{O}_{X}(\sum_{i}(m_{i}-1)F_{i}).$$
\end{thm}

The following theorem is the case of (2) in Remark \ref{class0}.
\begin{thm}\label{elliptic2}
Let $\pi: X\rightarrow C$ be elliptic fibration, $X$ is minimal of Kodaira dimension $1$. Assume $P_{g}(X)=0$, $L\neq \mathcal{O}_{C}$, $g(C)=1$, $q(X)=1$, and $d=0$, then $D(X)$ has no nontrivial
semi-orthogonal decompositions.
\end{thm}

\begin{proof}
According to Theorem \ref{canonicalbundle},
$$\omega_{X}=\pi^{\ast}L\otimes \mathcal{O}_{X}(\sum_{i}(m_{i}-1)F_{i}).$$
Since $\mathsf{deg}L=0$, as an algebraic class,
$$\omega_{X}\cong \mathcal{O}_{X}(\sum_{i}(m_{i}-1)F_{i}).$$
Clearly, the base locus of $\mathcal{O}_{X}(\sum_{i}(m_{i}-1)F_{i})$ is contained in finitely many fibers $m_{i}F_{i}$. The remaining argument is the same as the proof in \cite[Theorem 4.2]{kawatani2018nonexistence}.
\end{proof}

The remaining case is the case (1) in the Remark \ref{class0}, that is, $L$ is trivial, $P_{g}(X)=0$, $g(C)=0$, and $q(X)=1$. The canonical bundle formula becomes
$$\omega_{X}\cong \pi^{\ast}\mathcal{O}_{\mathbb{P}^{1}}(-2)\otimes \mathcal{O}(\sum_{i}(m_{i}-1)F_{i}).$$

If we move $\mathcal{O}(-1)$ to each factor $\mathcal{O}_{X}((m_{i}-1)F_{i})$, it would become a noneffective divisor. We can not say the general theorem for these cases. Since $q(X)=1$, it is still possible that the paracanonical systems admits nontrivial sections.

\begin{ques}
Let $\pi: X\rightarrow \mathbb{P}^{1}$ be an elliptic fibration, $X$ is minimal of Kodaira dimension~$1$. Suppose $P_{g}(X)=0$, and $q(X)=1$. Is $D(X)$ indecomposable ?
\end{ques}

\subsection{Minimal surfaces of general type, $P_{g}(X)=q(X)=1$, $K^{2}=2$ or $K^{2}=3$.}\ \\
Firstly, the new approach presented in this paper may work for more examples of minimal surfaces of general type. It is interesting to investigate in the future. We present some boundary cases of Conjecture \ref{newconj}, namely
$P_{g}(X)=q(X)=1$.

\begin{thm}\label{surfacegeneraltype}
Let $X$ be a minimal surface of general type such that $P_{g}(X)=q(X)=1$, $K^{2}=2$ or $K^{2}=3$, then $D(X)$ has no nontrivial semi-orthogonal decompositions.
\end{thm}

\begin{proof}
  For these types of surfaces, there are complete classifications, in particular, it was known that $\dim \mathsf{PBs}\vert\omega_{X}\vert= 0$ or $\mathsf{PBs}\vert\omega_{X}\vert= \varnothing\vert$ \cite[Section 5, Lemma 4.10]{CC}. Hence $D(X)$ has no nontrivial semi-orthogonal decompositions by Corollary \ref{emptynosod}.
\end{proof}

The proof in Theorem \ref{surfacegeneraltype} was used to classify the surfaces. We present a different proof for some explicit constructions of surfaces of general type such that $P_{g}(X)=q(X)=1$, $K^{2}=2$ or $K^{2}=3$.

\par

Let $E^{(2)}$ be the double symmetric product of an elliptic curve. Define $\pi: E^{(2)}\rightarrow E$ via $(x,y)\mapsto x+y$. We denote $D_{p}$ the class $(p,z-p)$ of $E^{(2)}$, and $F_{x}$ the fibre of $\pi$ at $x\in E$.
\par
The following proposition is for the case $K^{2}=2$.
\begin{prop}
   Let $B$ be a general member of $\vert-6D+2F\vert$ in $E^{(2)}$. Let $X'$ be a double cover of $E^{(2)}$ ramified over $B$. Let $X$ be the minimal resolution of $X'$, and $X'$ is isomorphic to the canonical model of $X$. Then
   $D(X)$ has no nontrivial semi-orthogonal decompositions.
\end{prop}

\begin{proof}
  These examples indeed exist, see \cite[Part III]{CC81} and \cite[Section 5]{CC}. It is clear that $K_{X'}$ is a line bundle. Since $K_{X'}\cong f^{\ast}(K_{E^{(2)}}+ B/2)$, and $K_{X}\cong g^{\ast}K_{X'}$, where $g$ is the resolution $g:X\rightarrow X'$. Write the composition of $g$ and $f$ as $h: X\rightarrow E^{(2)}$.
  Then
  $$K_{X}\cong h^{\ast}(K_{E^{(2)}}+ B/2).$$
  Since $K_{E^{(2)}}\cong^{alg} -2D+F$, therefore $K_{E^{(2)}}+ B/2\cong^{alg}D$. Since $\mathsf{PBs}\vert D\vert=\varnothing$, we have $\mathsf{PBs}\vert\omega_{X}\vert=\varnothing$. Thus, $D(X)$ has no nontrivial semi-orthogonal decompositions.
\end{proof}

\begin{rem}
   We actually obtain examples of surface $X'$ with Gorenstein singularities such that $Perf(X')$ has no nontrivial semi-orthogonal decompositions by Theorem \ref{singularsod}.
\end{rem}
Let $E^{(3)}$ be the triple symmetric product of $E$, similar notations for $D$ and $F$. The following proposition is for some examples such that $K^{2}=3$, see \cite[Theorem 5.8]{CC}.

\begin{prop}
Let $X'$ be a divisor in $E^{(3)}$. Assume $X'$ is homologous to $\vert4D-F\vert$ and has only simple singularities. Let $X$ be a minimal resolution of $X'$ where $X'$ is isomorphic to the canonical model of $X$. Then $D(X)$ has no nontrivial semi-orthogonal decompositions.
\end{prop}

\begin{proof}
  These examples indeed exist \cite[Theorem 5.8]{CC}. We prove that $\mathsf{PBs}\vert K_{X'}\vert=\varnothing$. By Adjunction,
  $$K_{X'}\cong K_{E^{(3)}}+X'\vert_{X'}\cong^{alg} -3D+F+4D-F\vert_{X'}=D\vert_{X'}.$$
  Clearly, $\mathsf{PBs}\vert \mathcal{O}_{X'}(D)\vert=\varnothing$.
\end{proof}

\begin{rem}
  It is also interesting to investigate the cases $4\leq K^{2}\leq 9$, which would be more complicated.
\end{rem}
\section{Application to Bridgeland stability condition}
\begin{thm}\label{boundedphase}
Let $X$ be a smooth projective variety, and $Z=\mathsf{PBs}\vert \omega_{X}\vert$.
Let $\sigma$ be a Bridgeland stability condition of $D(X)$. Take any $k(x)$, $x\in X\setminus Z$. Assume $k(x)$ is not semi-stable, then the phase number of $HN$ factors of $k(x): \phi_{1}>\phi_{2}>\cdots >\phi_{n}$ satisfies
$$\phi_{i}-\phi_{i+1}\leq \dim X-1.$$
\end{thm}

\begin{proof}
The line bundle $L\in Pic^{0}(X)$ preserves the $HN$ filtration of skycraper sheaves \cite[Corollary 3.5.2]{Polishchuk}. So it is enough to prove the theorem for the case $Z=\mathsf{Bs}\vert \omega_{X}\vert$. Let $s$ be a section of $\omega_{X}$ which does not vanish at $x\in X\setminus Z$. If $k(x)$ is semi-stable with respect to the stability condition, then there is nothing to prove. We assume that $k(x)$ is not semi-stable. Let

$$0=E_{0}\rightarrow E_{1}\rightarrow E_{2}\rightarrow \cdots \rightarrow E_{n-1}\rightarrow E_{n}=k(x)$$
be the $HN$ filtration of $k(x)$, $x\in U=X\setminus Z(s)$, with $HN$ factors $A_{i}:=\operatorname{Cone}(E_{i-1}\rightarrow E_{i})$ such that $A_{i}\in \mathcal{P}(\phi_{i})$,
$$\phi_{1}>\phi_{2}>\cdots >\phi_{n}.$$
Define $K_{i}=\operatorname{Cone}(E_{i}\rightarrow k(x))$.
There is a commutative diagram.
 $$\xymatrix{E_{i}\ar[r]^{id}\ar[d]&E_{i}\ar[r]\ar[d]&0\ar[d]\\
 E_{i+1}\ar[r]\ar[d]&k(x)\ar[r]\ar[d]&K_{i+1}\ar[d]^{id}\\
 A_{i+1}\ar[r]&K_{i}\ar[r]&K_{i+1}}$$
Consider composition $K_{i}\rightarrow E_{i}[1]\rightarrow E_{i}[1]\otimes \omega_{X}$ with section $s$,
$$Hom(K_{i}, E_{i}[1]\otimes \omega_{X})\cong Hom(E_{i},K_{i}[\dim X-1])^{\vee}.$$
If $\phi^{-}(E_{i})>\phi^{+}(K_{i})+\dim X-1$, then $Hom(K_{i},E_{i}[1]\otimes \omega_{X})=0$. Therefore $k(x)\vert_{U}\cong E_{i}\vert_{U}\oplus K_{i}\vert_{U}$. Since $Hom(k(x), k(x))=\mathbb{C}$, $E_{i}\vert_{U}$ or $K_{i}\vert_{U}$ must be zero. Suppose $E_{i}\vert_{U}=0$, it implies that the support of $E_{i}$ is in $Z(s)$, then $Hom^{\ast}(E_{i}, k(x))=0$. Hence $K_{i}\cong k(x)\oplus E_{i}[1]$. Since $Hom(E_{i}[1],K_{i})=0$ by the assumption that $\phi^{-}(E_{i})> \phi^{+}(K_{i})+\dim X - 1$, we have $E_{i}\cong0$, a contradiction. Similarly, suppose $K_{i}\vert_{U}=0$, then $E_{i}\cong k(x)\oplus K_{i}[-1]$, which implies $K_{i}\cong 0$.
So we prove $\phi^{-}(E_{i})\leq\phi^{+}(K_{i})+\dim X-1$. Clearly $\phi^{-}(E_{i})=\phi_{i}$. We prove by induction that $\phi^{+}(K_{i})\leq \phi_{i+1}$.
\begin{lem}
$\phi^{+}K_{i}\leq \phi_{i+1}$, $i= 1,\ 2,\ \cdots,\ n-1$.
\end{lem}

 \begin{proof}
 Consider the triangle $E_{n-1}\rightarrow k(x)\rightarrow K_{n-1}$, then $K_{n-1}=A_{n}$ by definition. Clearly
 $\phi^{+}K_{n-1}=\phi(A_{n})=\phi_{n}$. Now assume $\phi^{+}(K_{i+1})\leq \phi_{i+2}$, we need to prove that $\phi^{+}(K_{i})\leq \phi_{i+1}$. Let $B_{1}$ be a $HN$ factor of $K_{i}$ such
 that $\phi(B_{1})=\phi^{+}(K_{i})$. Suppose $\phi(B_{1})>\phi_{i+1}$, then $Hom(B_{1}, A_{i+1})=0$, and $Hom(B_{1}, K_{i+1})=0$. Therefore $Hom(B_{1}, K_{i})=0$ by the traingle
 $$A_{i+1}\rightarrow K_{i}\rightarrow K_{i+1}\rightarrow A_{i+1}[1].$$
 But if we apply $Hom(B_{1}, \bullet)$ to the HN filtration of $K_{i}$, we get
 $$Hom(B_{1}, B_{1})\cong Hom(B_{1}, K_{i})=0,$$
 a contradiction. Thus $\phi^{+}(K_{i})\leq \phi_{i+1}$.
 \end{proof}
According to the lemma above, we have inequality $\phi_{i}-\phi_{i+1}\leq \dim X-1$.
\end{proof}

\begin{cor}
Let $X=S^{i}C$, $g(C)\geq 2$, and $i\leq g(C)-1$, or be a minimal surface of general type with $P_{g}(X)=q(X)=1$, $K^{2}=2$ or $K^{2}=3$. If $\sigma$ is any Bridgeland stability condition of $D(X)$, then for any closed point $x\in X$, the phase number of $HN$ factors of skyscraper sheaf $k(x)$ (if $k(x)$ is not semi-satble): $\phi_{1}>\phi_{2}>\cdots >\phi_{n}$ satisfies
$$\phi_{i}-\phi_{i+1}\leq \dim X-1.$$
\end{cor}

\end{document}